\documentclass[12pt]{amsart}

\usepackage[margin=1in]{geometry}
\usepackage{amsthm}
\usepackage{enumitem}  
\usepackage[OT2,T1]{fontenc}
\usepackage{amscd, amssymb, enumitem, amsmath, mathrsfs, amsfonts}
\usepackage{url}
\usepackage{tikz}
\usepackage{fullpage}
\usepackage{microtype}
\usetikzlibrary{decorations.pathreplacing}
\usepackage[backref=page]{hyperref}
\usepackage{hyperref}
\usepackage[utf8]{inputenc}
\usepackage{comment}

\makeatletter
\@namedef{subjclassname@2020}{%
  \textup{2020} Mathematics Subject Classification}
\makeatother

\DeclareSymbolFont{cyrletters}{OT2}{wncyr}{m}{n}
\DeclareMathSymbol{\Sha}{\mathalpha}{cyrletters}{"58}

\newtheorem{theorem}{Theorem}[section]

\newtheorem{lemma}[theorem]{Lemma}

\newtheorem*{proposition*}{Proposition}

\theoremstyle{definition}

\newtheorem{example}[theorem]{Example}

\newtheorem{remark}[theorem]{Remark}

\newtheorem*{acknowledgement}{Acknowledgements}

\theoremstyle{remark}

\title{Reduction and isogenies of elliptic curves}
\author{Mentzelos Melistas}

\address{Charles University, Faculty of Mathematics and Physics, Department of
Algebra, Sokolov\-sk\' a 83, 18600 Praha~8, Czech Republic}

\address{University of Twente, Department of Applied Mathematics, Drienerlolaan 5, 7522 NB Enschede, The Netherlands}

\begin{document}

\maketitle

\begin{abstract}
    Let $R$ be a complete discrete valuation ring with fraction field $K$ and perfect residue field $k$ of characteristic $p>0$. Let $E/K$ be an elliptic curve with a $K$-rational isogeny of prime degree $\ell$. In this article, we study the possible Kodaira types of reduction that $E/K$ can have. We also prove some related results for elliptic curves over $\mathbb{Q}$.
\end{abstract}

\section{Introduction}

Let $K$ be a number field and let $E/K$ be an elliptic curve. A $K$-rational isogeny $\phi$ of $E/K$ is an isogeny $\phi: \; E \longrightarrow E'$ which is defined over $K$, for some elliptic curve $E'/K$. The study of $K$-rational isogenies of elliptic curves (and their possible degrees) for different number fields $K$ is a rich topic with a long history (see e.g. \cite{mazur}, \cite{momose}, \cite{larsonvaintrob}, \cite{michaudjacobs}). In this paper we are interested in answering the following related question; Given an elliptic curve $E/K$ with a cyclic $K$-rational isogeny of prime degree, then can we say anything about the reduction properties of $E/K$? To be more precise, if $E/K$ is an elliptic curve with a $K$-rational isogeny of prime degree $\ell > 3$, then we are interested in the possible Kodaira types of reduction that can occur. The reduction properties of elliptic curves with complex multiplication and of elliptic curves with torsion points have been previously studied by the author in \cite{mentzeloscmkodaratypes} and \cite{mentzeloskodairaandtorsion}, respectively.

Since determining the reduction type of an elliptic curve is a problem of local nature, we can consider elliptic curves over complete discrete valuation rings. Our main result is the following theorem, proved in the next section.

\begin{theorem}\label{theorem1}
    Let $R$ be a complete discrete valuation ring with fraction field $K$ and perfect residue field $k$ of characteristic $p > 3$. Let $E/K$ be an elliptic curve with a $K$-rational isogeny of prime degree $\ell > 3$ such that $p \neq \ell$.
    \begin{enumerate}
        \item If $\ell-1 \equiv 2 \text{ or } 10 \: (\text{mod }12)$, then $E/K$ has either semi-stable reduction or reduction of type \textup{I}$_n^*$ for some $n \geq 0$.
        \item If $\ell-1 \equiv 4 \text{ or } 8 \: (\text{mod }12)$, then $E/K$ has either semi-stable reduction or reduction of type \textup{III}, \textup{III}$^*$, or \textup{I}$_n^*$ for some $n \geq 0$.
        \item If $\ell-1 \equiv 6 \: (\text{mod }12)$, then $E/K$ cannot have reduction of type \textup{III} or \textup{III}$^*$.
    \end{enumerate}
\end{theorem}

We present some examples (see Examples \ref{example1}, \ref{example2}, as well as the paragraph before them) showing that all Kodaira types that appear in Parts $(i)$ and $(ii)$ Theorem \ref{theorem1} do indeed occur. We also prove a partial analog (see Theorem \ref{theorem3} below) of Theorem \ref{theorem1} when $p=2$ or $3$ and we explain, in Remark \ref{remarkorder23}, why the restriction that $\ell >3$ is natural in the context of Theorem \ref{theorem1}. We note that Theorem \ref{theorem1} is false when $\ell =p$. Indeed, in Example \ref{example5overq} below we present examples of elliptic curves $E/\mathbb{Q}$ with a $\mathbb{Q}$-rational isogeny of degree $\ell=5$ that have modulo $5$ reduction of type II, II$^*$, III, III$^*$, IV, IV$^*$, I$_0^*$, I$_1^*$, and I$_1$.

We now turn our attention to elliptic curves over $\mathbb{Q}$. In this case, a celebrated theorem of Mazur (see \cite[Theorem 1]{mazur}) provides a classification for the possible prime degrees of $\mathbb{Q}$-rational isogenies. Namely, if $E/\mathbb{Q}$ is an elliptic curve with a $\mathbb{Q}$-rational isogeny of prime degree $\ell$, then $\ell \in \{2,3,5,7,11,13,17, 19, 37, 43, 67, 163  \}$. A selection of our results from Section \ref{sectionoverq} concerning elliptic curves over $\mathbb{Q}$ is the following theorem (See Theorems \ref{theorem6}, \ref{theorem5}. \ref{theorem8}, and \ref{theorem7} below).

   \begin{theorem}\label{theorem4}
        Let $E/\mathbb{Q}$ be an elliptic curve with a $\mathbb{Q}$-rational isogeny of prime degree $\ell$.
        \begin{enumerate}
        \item If $\ell=11, 19, 43, 67,$ or $163$, and $p$ is a prime such that $p \neq 2, \ell$, then $E/\mathbb{Q}$ has either good reduction or reduction of type \textup{I}$_0^*$ modulo $p$.
        \item If $\ell=19, 43, 67,$ or $163$, then $E/\mathbb{Q}$ has reduction of type \textup{III} or \textup{III}$^*$ modulo $\ell$.
        \item If $\ell=11, 19, 37, 43, 67,$ or $163$, then $E/\mathbb{Q}$ has either good reduction or reduction of type \textup{II} or \textup{II}$^*$ modulo $2$.
        \item If $\ell=17$ or $37$ and $p \neq 2, 5, 7,  17$ is a prime, then $E/\mathbb{Q}$ has either good reduction or reduction of type \textup{III}, \textup{III}$^*$, or \textup{I}$_0^*$ modulo $p$.
        \end{enumerate}
    \end{theorem}

    For the prime numbers $\ell$ treated in Theorem \ref{theorem4} we also classify the possible reduction types modulo $2$ and modulo $\ell$ in Section \ref{sectionoverq}. Moreover, by following the proof of each part of Theorem \ref{theorem4}, which involves the computation of the possible Kodaira types of elliptic curves with a fixed $j$-invariant, we can see that in fact all allowed Kodaira types do indeed occur.

    This article is organized as follows. In Section \ref{section2}, after recalling some background material, we prove Theorem \ref{theorem1}. Then we present some examples showing that all Kodaira types that appear in Parts $(i)$ and $(ii)$ of Theorem \ref{theorem1} do indeed occur. Finally, Section \ref{sectionoverq} is devoted to elliptic curves over $\mathbb{Q}$ and Theorem \ref{theorem4} is proved.

\begin{acknowledgement}
I would like to thank the anonymous referee for carefully reading the manuscript and for providing several very useful corrections. The author was supported by Czech Science Foundation (GA\v CR) grant 21-00420M and by Charles University Research Center program No.UNCE/SCI/022.
\end{acknowledgement}

\section{Proof of Theorem \ref{theorem1}}\label{section2}

 Let $R$ be a complete discrete valuation ring with valuation $v$, fraction field $K$, and perfect residue field $k$ of characteristic $p>0$. Let $K^{\text{s}}$ be a fixed separable closure of $K$ and let $G_K=\text{Gal}(K^{\text{s}}/K)$. Assume that $E/K$ has a cyclic $K$-rational isogeny $\phi$ of prime degree $\ell > 3$ with kernel denoted by $C$. We assume that $\ell \neq p$ throughout this section. Let $P \in E[\ell]$ be a generator for $C$. Write $L/K$ for the minimal field of definition of the point $P$, i.e., $L$ is the field obtained by adjoining the coordinates of $P$ to the field $K$. Extend $P$ to a basis $\{P, Q \}$ of $E[\ell]$ and denote $\bar{\rho}_{E,\ell} : G_K \longrightarrow \rm{GL}(\mathbb{F}_{\ell})$ the mod $\ell$ representation of $E/K$ with respect to the basis $\{ P, Q \}$. Let $B$ be the Borel subgroup of $\rm{GL}(\mathbb{F}_{\ell})$, i.e.,
$$B=\left\{ \begin{pmatrix}
    a & b \\
    0 & d
\end{pmatrix} : \: a,b,d \in \mathbb{F}_{\ell} \text{ and } ad \neq 0 \right\},$$

and let $B_1$ be the subgroup 
$$B_1 = \left\{ \begin{pmatrix}
    1 & b \\
    0 & d
\end{pmatrix} : \: b,d \in \mathbb{F}_{\ell} \text{ and } d \neq 0\right\}.$$

Using Galois theory we can prove the following (see also \cite[Lemma 3.1]{cremonanajman}). 
\begin{lemma}\label{lemmafieldofdfinition}
      Let $E/K$ be an elliptic curve with a cyclic $K$-rational isogeny $\phi$ of prime degree $\ell > 3$ such that $\ell \neq p$. Let $P \in E[\ell]$ be a generator for the kernel of $\phi$ and write $L/K$ for the minimal field of definition of $P$. Then the degree of the extension $L/K$ divides $\ell-1$.
\end{lemma}
\begin{proof}
      Since the isogeny $\phi$ is defined over $K$, we have that $\bar{\rho}_{E,\ell}(G_K)$ is a subgroup of $B$. Therefore, it follows from Galois theory that $$[L:K]=[B \cap\bar{\rho}_{E,\ell}(G_K): B_1 \cap \bar{\rho}_{E,\ell}(G_K)]=[\bar{\rho}_{E,\ell}(G_K): B_1 \cap \bar{\rho}_{E,\ell}(G_K)],$$ which divides $[B : B_1]=\ell-1$.

\end{proof}

     Denote by $E_L/L$ the base extension of $E/K$ to $L$. The following lemma will be useful in our proofs below.
     
     \begin{lemma}\label{lemmasemistable}
         Let $E/K$ be an elliptic curve with a cyclic $K$-rational isogeny $\phi$ of prime degree $\ell > 3$ such that $\ell \neq p$. Let $P \in E[\ell]$ be a generator for the kernel of $\phi$ and write $L/K$ for the minimal field of definition of $P$. Then the curve $E_L/L$ has semi-stable reduction.
    \end{lemma}
    \begin{proof}
        Assume that $E_L/L$ does not have semi-stable reduction and we will find a contradiction. Consequently, we assume from now on that $E_L/L$ has additive reduction. Let $R_L$ be the integral closure of $R$ in $L$, which is again a discrete valuation ring because $R$ is complete, and denote by $k_L$ its residue field. Pick a minimal Weierstrass equation for $E_L/L$ and denote by $\widetilde{E_L}/k_L$ the corresponding reduction. Denote also by $(E_L)_0(L)$ the set of points with nonsingular reduction and by $(E_L)_1(L)$ the kernel of the reduction map.

        It follows from \cite[Proposition VII.2.1]{aec}, that there exists a short exact sequence of abelian groups
        $$0 \longrightarrow (E_L)_1(L) \longrightarrow (E_L)_0(L) \longrightarrow (\widetilde{E_L})_{\rm{ns}}(k_L) \longrightarrow 0,$$ where $(\widetilde{E_L})_{\rm{ns}}(k_L)$ is the set of non-singular points of $\widetilde{E_L}/k_L$ and the right-hand map is the reduction map. Consider now the point $P \in E_L(L)$, which has order $\ell$. We will first show that $P \not\in (E_L)_0(L)$. Suppose that $P \in (E_L)_0(L)$, and we will find a contradiction. Since $\ell$ is coprime to $p$ and $P$ has order $\ell$, we find, using \cite[Proposition VII.3.1]{aec}, that $P \not\in (E_L)_1(L)$. Therefore, if $P \in (E_L)_0(L)$, then we must have that the reduction of $P$ must have order $\ell$ in $(\widetilde{E_L})_{\rm{ns}}(k_L)$. However, by \cite[Exercise III.3.5]{aec} we see that $(\widetilde{E_L})_{\rm{ns}}(k_L)$ is the additive group $\mathbb{G}_{\rm{a}}(k_L)$, where $\mathbb{G}_{\rm{a}}/k_L$ is the additive group scheme over $k_L$. Since $\mathbb{G}_{\rm{a}}(k_L)$ has no points of order $\ell$, we see that $P \not\in (E_L)_0(L)$.

        Finally, if $E_L/L$ has additive reduction, then, using \cite[Corollary IV.9.2]{silverman2} (or \cite{tatealgorithm}), we find that the group $E_L(L)/(E_L)_0(L)$ has order at most $4$. However, since $P \not\in (E_L)_0(L)$, we must have that $E_L(L)/(E_L)_0(L)$ has order divisible by the prime $\ell$, which is bigger than $3$. This is a contradiction and, hence, $E_L/L$ has semi-stable reduction. This completes the proof of our claim.

    \end{proof}

\begin{theorem}\label{theorem2}
    Let $R$ be a complete discrete valuation ring with valuation $v$, fraction field $K$, and perfect residue field $k$ of characteristic $p>0$. Let $E/K$ be an elliptic curve with potentially good reduction and a $K$-rational isogeny of prime degree $\ell > 3$ with $\ell \neq p$. Denote by $\Delta_{E/K}$ the discriminant of a minimal Weierstrass equation for $E/K$. Then $$12 \text{ divides } (\ell-1)v(\Delta_{E/K}).$$
\end{theorem}

\begin{proof}

     Assume that $E/K$ has a cyclic $K$-rational isogeny of degree $\ell$ with kernel denoted by $C$. Let $P \in E[\ell]$ be a generator for $C$. Write $L/K$ for the minimal field of definition of the point $P$. 
    
     Let $R_L$ be the integral closure of $R$ in $L$, which is again a discrete valuation ring because $R$ is complete. We denote by $v_L$ the associated (normalized) valuation of $R_L$. Note that the restriction $v_L|K$ of $v_L$ to $K$ satisfies $v_L|_K=ev$ where $e$ is the ramification index of $L/K$. By Lemma \ref{lemmasemistable} the curve $E_L/L$ has semi-stable reduction. Therefore, since we assume that $E/K$ has potentially good reduction, we find that $E_L/L$ has good reduction. Thus, if $\Delta_{E_L/L}$ is the discriminant of a minimal Weierstrass equation of $E_L/L$, then we must have that $v_L(\Delta_{E_L/L})=0$. 
     
     On the other hand, $\Delta_{E/K}$ is the discriminant of a (not necessarily minimal) Weierstrass equation for $E_L/L$. Since when we perform a change of variable the valuation of the discriminant changes by a factor of $12$, we see that  $12$ divides $v_L(\Delta_{E_L/L})-v_L(\Delta_{E/K})$. However, from the previous paragraph we have that $v_L(\Delta_{E_L/L})=0$ and, hence, $12$ divides $v_L(\Delta_{E/K})=ev(\Delta_{E/K})$. Moreover, it follows from Lemma \ref{lemmafieldofdfinition} that the degree of the extension $L/K$ divides $\ell-1$. Therefore, we find that $e$ divides $\ell-1$ and, hence, we see that $12$ divides $(\ell-1)v(\Delta_{E/K}).$
     \end{proof}

We are now ready to proceed to the proof of Theorem \ref{theorem1}.
 \begin{proof}[Proof of Theorem \ref{theorem1}]
  Assume that $E/K$ has a cyclic $K$-rational isogeny of degree $p$ with kernel denoted by $C$. Let $P \in E[\ell]$ be a generator for $C$. Write $L/K$ for the minimal field of definition of the point $P$. Lemma \ref{lemmafieldofdfinition} tells us that $[L:K]$ divides $\ell-1$ while Lemma \ref{lemmasemistable} tells us that the base extension $E_L/L$ of $E/K$ to $L$ has semi-stable reduction.

   If $E_L/L$ has multiplicative reduction, then using Tate's algorithm \cite{tatealgorithm}, since $p  >3$, we find that that $E/K$ has either multiplicative reduction or reduction of type I$_{n}^*$, for some $n > 0$. We assume from now on that $E_L/L$ has good reduction. 
   
   {\it Proof of $(i)$:} Assume that $\ell-1 \equiv 2 \text{ or } 10 \: (\text{mod }12)$. Denote by $\Delta_{E/K}$ the discriminant of a minimal Weierstrass equation for $E/K$. Theorem \ref{theorem2} tells us that $12 \text{ divides } (\ell-1)v(\Delta_{E/K}).$ Since $\ell-1 \equiv 2 \text{ or } 10 \: (\text{mod }12)$, we find that $$0 \equiv (\ell-1)v(\Delta_{E/K})\equiv 2v(\Delta_{E/K}) \text{ or } 10v(\Delta_{E/K}) \: (\text{mod }12).$$ Since $p>3$, this is only possible when $v(\Delta_{E/K})=0$ or $6$. Therefore, using \cite[Page 365]{silverman2}, we see that $E/K$ has either good reduction or reduction of type I$_0^*$. This proves part $(i)$.

   {\it Proof of $(ii)$:} Assume that $\ell-1 \equiv 4 \text{ or } 8 \: (\text{mod }12)$. Denote by $\Delta_{E/K}$ the discriminant of a minimal Weierstrass equation for $E/K$. Theorem \ref{theorem2} tells us that $12 \text{ divides } (\ell-1)v(\Delta_{E/K}).$ Since $\ell-1 \equiv 4 \text{ or } 8 \: (\text{mod }12)$, we find that $$0 \equiv (\ell-1)v(\Delta_{E/K})\equiv 4v(\Delta_{E/K}) \text{ or } 8v(\Delta_{E/K}) \: (\text{mod }12).$$ Since $p>3$, this is only possible when $v(\Delta_{E/K})=0, 3, 6, \text{ or } 9$. Therefore, using \cite[Page 365]{silverman2}, we find that $E/K$ has either good reduction or reduction of type III, III$^*$, or I$_0^*$. This proves part $(ii)$. 

   {\it Proof of $(iii)$:} Assume now that $\ell-1 \equiv 6 \: (\text{mod }12)$. Denote by $\Delta_{E/K}$ the discriminant of a minimal Weierstrass equation for $E/K$. Theorem \ref{theorem2} tells us that $12 \text{ divides } (\ell-1)v(\Delta_{E/K}).$  Since $\ell-1 \equiv 6 \: (\text{mod }12)$, we find that $$0 \equiv (\ell-1)v(\Delta_{E/K})\equiv 6v(\Delta_{E/K}) \: (\text{mod }12).$$ From this we obtain that $v(\Delta_{E/K}) \neq 3 \text{ or } 9$. Therefore, using \cite[Page 365]{silverman2}, we find that $E/K$ cannot have reduction of type III or III$^*$. This completes the proof of our theorem.
 \end{proof}

 \begin{remark}\label{remarkorder23}
    We explain in this remark why the restriction that $\ell >3$ is natural in the context of Theorem \ref{theorem1}. First, for an elliptic curve $E/K$ having a $K$-rational isogeny of degree $2$ is the same as having a $K$-rational torsion point of order $2$. Thus, studying elliptic curves with an isogeny of degree $2$ is the same as studying elliptic curves with a $K$-rational point of order $2$.
    
    On the other hand, it is not hard to show that if an elliptic curve $E/K$ has a $K$-rational isogeny of degree $3$, then a quadratic twist of $E/K$ has a $K$-rational point of order $3$ (see also \cite[Exercise 2.6]{sikseknotes}). Therefore, studying elliptic curves with an isogeny of degree $3$ is the same as studying elliptic curves whose twists have a $K$-rational point of order $3$. We note that when the absolute ramification index of $K$ is $1$, the possible Kodaira types of reduction of elliptic curves $E/K$ that have a $K$-rational point of order $3$ have been described by Kozuma in \cite[Proposition 3.5]{koz} and \cite[Lemma 3.6]{koz}.
\end{remark}

Let now $K$ be a number field and let $E/K$ be an elliptic curve that has a $K$-rational isogeny of prime degree $\ell > 3$. Let $\mathfrak{p}$ be a prime of $K$ which lies above a rational prime $p > 3$ with $\ell \neq p$. Assume that $E/K$ has reduction of Kodaira type I$_n$, for some $n\geq 0$, modulo $\mathfrak{p}$. By performing an appropriate quadratic twist we can construct an elliptic curve $E'/K$ with a $K$-rational isogeny of degree $\ell$ and reduction of Kodaira type I$_n^*$ modulo $\mathfrak{p}$ (see \cite{com} for background on Kodaira types of quadratic twists). Thus, the Kodaira types I$_n^*$ that appear in Parts $(i)$ and $(ii)$ of Theorem \ref{theorem1} do indeed occur.

The following two examples show that the Kodaira types III and III$^*$ allowed by Part $(ii)$ of Theorem \ref{theorem1} also occur.
\begin{example}\label{example1}
    Consider the elliptic curve $E/\mathbb{Q}$ given by the following Weierstrass equation $$E\: : \: y^2+xy+y=x^3-190891x-36002922.$$ This curve has LMFDB \cite{lmfdb} label \href{https://www.lmfdb.org/EllipticCurve/Q/14450/b/1}{14450.b1}. Using LMFDB it is easy to see that $E/\mathbb{Q}$ has a $\mathbb{Q}$-rational isogeny of degree $17$ and that it has reduction of Kodaira type III modulo $5$.
\end{example}

\begin{example}\label{example2}
    Consider the elliptic curve $E/\mathbb{Q}$ given by the following Weierstrass equation $$E\: : \: y^2+xy=x^3-16513x-916983.$$ This curve has LMFDB \cite{lmfdb} label \href{https://www.lmfdb.org/EllipticCurve/Q/14450/w/2}{14450.w2} and is a quadratic twist of the elliptic curve with label \href{https://www.lmfdb.org/EllipticCurve/Q/14450/b/1}{14450.b1} considered in the previous example. It is easy to see that $E/\mathbb{Q}$ has a $\mathbb{Q}$-rational isogeny of degree $17$ and that it has reduction of Kodaira type III$^*$ modulo $5$.
\end{example}

The following example illustrates two important aspects related to Theorem \ref{theorem1}. Firstly, the assumption that $\ell \neq p$ in Theorem \ref{theorem1} is necessary as more reduction types can occur. Secondly, when $\ell-1 \equiv 6 \: (\text{mod }12)$ (as is the case for $\ell=19$ below) then the reduction types II and IV$^*$ can indeed occur.

\begin{example}
  
  Consider the elliptic curve $E/\mathbb{Q}(\sqrt{-3})$ given  by the following Weierstrass equation $$E\: : \: y^2+xy+y=x^3+(184a-12)x+101a+872,$$ where $a=\frac{1+\sqrt{-3}}{2}$. This curve has LMFDB \cite{lmfdb} label \href{https://www.lmfdb.org/EllipticCurve/2.0.3.1/61009.7/b/1}{2.0.3.1-61009.7-b1} and has a $\mathbb{Q}(\sqrt{-3})$-rational isogeny of order $19$. Denote by $\mathfrak{p}$ and $\mathfrak{q}$ the prime ideals $(4a-3)$ and $(-5a+3)$ of the ring of integers of $\mathbb{Q}(\sqrt{-3})$, respectively. Note that $\mathfrak{p}$ lies above $13$ and $\mathfrak{q}$ lies above $19$. Using the database it is easy to see that $E/\mathbb{Q}(\sqrt{-3})$ has a $\mathbb{Q}(\sqrt{-3})$-rational isogeny of degree $19$, bad reduction of Kodaira type IV$^*$ modulo $\mathfrak{p}$, and reduction of Kodaira type III modulo $\mathfrak{q}$.

  Let $d_1=4a-3$ and $d_2=-5a+3$. It follows from \cite[Proposition 1]{com} that the quadratic twist $E^{d_1}/\mathbb{Q}(\sqrt{-3})$ of $E/\mathbb{Q}(\sqrt{-3})$ has bad reduction of Kodaira type  II modulo $\mathfrak{p}$. Moreover, it follows from \cite[Proposition 1]{com} that the quadratic twist $E^{d_2}/\mathbb{Q}(\sqrt{-3})$ of $E/\mathbb{Q}(\sqrt{-3})$ has reduction of Kodaira type  III$^*$ modulo $\mathfrak{q}$.  
\end{example}

\begin{theorem}\label{theorem3}
    Let $R$ be a complete discrete valuation ring with valuation $v$, fraction field $K$ of characteristic $0$, and perfect residue field $k$ of characteristic $p>0$. Let $E/K$ be an elliptic curve with a $K$-rational isogeny of prime degree $\ell > 3$. Assume that $v(p)=1$.
    \begin{enumerate}
        \item If $\ell-1 \equiv 2, 4, 8, \text{ or } 10 \: (\text{mod }12)$ and $p=2$, then $E/K$ cannot have reduction of type \textup{IV} or \textup{IV}$^*$.
        \item If $\ell-1 \equiv 2 \text{ or } 10 \: (\text{mod }12)$ and $p=3$, then $E/K$ has either semi-stable reduction, reduction of type \textup{IV} or \textup{II}$^*$, or reduction of type \textup{I}$_n^*$ for some $n \geq 0$.
    \end{enumerate}
\end{theorem}
\begin{proof}
    The proof is similar to the proof of Theorem \ref{theorem1}, using \cite{pap} instead of \cite[Page 365]{silverman2}. We include all the details here for completeness. Assume that $E/K$ has a cyclic $K$-rational isogeny of degree $p$ with kernel denoted by $C$. Let $P \in E[\ell]$ be a generator for $C$. Write $L/K$ for the minimal field of definition of the point $P$. Exactly as in the proof of Theorem \ref{theorem1}, Lemma \ref{lemmafieldofdfinition} tells us that $[L:K]$ divides $\ell-1$. Moreover, Lemma \ref{lemmasemistable} tells us that the base extension $E_L/L$ of $E/K$ to $L$ has semi-stable reduction.  We denote by $\Delta_{E/K}$ the discriminant of a fixed minimal Weierstrass equation for $E/K$.

     {\it Proof of $(i)$:} We assume for contradiction that $E/K$ has reduction of type IV or IV$^*$. This implies that $E_L/L$ has good reduction. Assume that $\ell-1 \equiv m \: (\text{mod }12)$, where $m \in \{ 2,4,8,10 \}$. Theorem \ref{theorem2} tells us that $12 \text{ divides } (\ell-1)v(\Delta_{E/K}).$ Since $\ell-1 \equiv m \: (\text{mod }12)$, we find that $$0 \equiv (\ell-1)v(\Delta_{E/K})\equiv mv(\Delta_{E/K}) \: (\text{mod }12).$$ On the other hand, since $v(2)=1$ and $E/K$ has reduction of type IV or IV$^*$, by \cite[Tableau IV]{pap} we have that $v(\Delta_{E/K})=4$ or $8$, which is a contradiction.

     {\it Proof of $(ii)$:} Proceeding exactly as in part $(i)$ we find that $$0 \equiv (\ell-1)v(\Delta_{E/K})\equiv 2v(\Delta_{E/K}) \text{ or } 10v(\Delta_{E/K}) \: (\text{mod }12).$$ Therefore, since $v(3)=1$, by \cite[Tableau II]{pap} we see that $E/K$ has either semi-stable reduction, reduction of type IV or II$^*$, or reduction of type I$_n^*$ for some $n \geq 0$. This proves our theorem.
\end{proof}

    We end this section by explaining why an analog of Theorem \ref{theorem1} for $\ell=p$ does not seem to exist. Concerning the case where the characteristic of the field $K$ is $0$, even for $K=\mathbb{Q}$ such a pattern does not seem to hold. This is because in Example \ref{example5overq} below among other examples we exhibit elliptic curves $E/\mathbb{Q}$ with a $\mathbb{Q}$-rational isogeny of degree $\ell=5$ that have modulo $5$ reduction of type II, II$^*$, III, III$^*$, IV, IV$^*$, I$_0^*$, I$_1^*$, and I$_1$. Thus, we do not see any pattern concerning their reduction types modulo $5$.

    \begin{example}\label{example5overq}
         Consider the curves with LMFDB labels \href{https://www.lmfdb.org/EllipticCurve/Q/75/a/2}{75.a2}, \href{https://www.lmfdb.org/EllipticCurve/Q/50/b/1}{50.b1}, \href{https://www.lmfdb.org/EllipticCurve/Q/175/a/2}{175.a2}, \href{https://www.lmfdb.org/EllipticCurve/Q/150/a/1}{150.a1}, \href{https://www.lmfdb.org/EllipticCurve/Q/50/a/1}{50.a1}, \href{https://www.lmfdb.org/EllipticCurve/Q/50/a/2}{50.a2}, \href{https://www.lmfdb.org/EllipticCurve/Q/275/b/1}{275.b1}, \href{https://www.lmfdb.org/EllipticCurve/Q/550/f/1}{550.f1}, and \href{https://www.lmfdb.org/EllipticCurve/Q/110/b/1}{110.b1}. Those curves have $\mathbb{Q}$-rational isogeny of degree $5$ and reduction modulo $5$ of type II, II$^*$, III, III$^*$, IV, IV$^*$, I$_0^*$, I$_1^*$, and I$_1$, respectively.
    \end{example}

    Suppose now that the characteristic of $K$ is $p$. Let $E/K$ be any elliptic curve. Extending scalars using the absolute Frobenius $Fr: \text{Spec}(K) \longrightarrow \text{Spec}(K)$, we obtain an elliptic curve $E^{(p)}/K$ and a purely inseparable isogeny $F: E \longrightarrow E^{(p)}$ of degree $p$. Thus every elliptic curve defined over $K$ has an isogeny of degree $p$, and, hence, we cannot have any restrictions on the reduction properties of elliptic curves with an isogeny of degree $p$. To remedy this problem one could restrict to separable isogenies. However, we note that given any isogeny $\phi$ of degree $p$ and dual isogeny $\hat{\phi}$, the facts that $\phi \circ \hat{\phi}=[p]$ and that $[p]$ is inseparable in characteristic $p$ combined imply that either $\phi$ or $\hat{\phi}$ is inseparable.

    \section{Elliptic curves over $\mathbb{Q}$}\label{sectionoverq}

     In this section, we focus on elliptic curves over $\mathbb{Q}$ and we prove Theorem \ref{theorem4}. Before we proceed to our proofs we briefly explain our general strategy. A similar strategy has been employed by Trbovi\'c in \cite{trbovic} to compute Tamagawa numbers of elliptic curves with isogenies. Let $\ell \geq 11$ be a prime and consider the modular curve $X_0(\ell)/\mathbb{Q}$ parametrizing elliptic curves together with an isogeny of degree $\ell$ (see \cite{katzmazur} and \cite{shimurabook} for general background on modular curves). In \cite[Table 4]{lozanorobledo}, we can find the $j$-invariants corresponding to non-cuspidal $\mathbb{Q}$-rational points of $X_0(\ell)/\mathbb{Q}$, i.e., the $j$-invariants of elliptic curves defined over $\mathbb{Q}$ that have a $\mathbb{Q}$-rational isogeny of degree $\ell$. 
     
     Moreover, according to \cite[Corollary X.5.4.1]{aec} all elliptic curves having the same $j$-invariant are twists of each other. Since all these $j$-invariants coming from \cite[Table 4]{lozanorobledo}, are not equal to $0$ or $1728$, we need to consider only quadratic twists. Finally, we will use results on reduction types of quadratic twists of elliptic curves. 
     
     If $E/\mathbb{Q}$ is an elliptic curve and $d$ is a square-free integer, then we will denote by $E^d/\mathbb{Q}$ the quadratic twist of $E/\mathbb{Q}$ by $d$. We recall now some of the results from \cite{com} for future reference. 
     
     \begin{lemma}\label{results1comalada}(See \cite[Proposition 1]{com}) Let $E/\mathbb{Q}$ be an elliptic curve and $d$ a square-free integer. If $p \neq 2$ is a prime with $p \mid d$, then the reduction types of $E/\mathbb{Q}$ and $E^d/\mathbb{Q}$ modulo $p$ are related as follows 
     \begin{center}
    \begin{tabular}{ ||c|c || } 
    \hline \hline
    Reduction type of $E/\mathbb{Q}$ modulo $p$ & Reduction type of $E^d/\mathbb{Q}$ modulo $p$ \\ 
    \hline
    \hline
     \textup{I}$_0$ & \textup{I}$_0^*$ \\ 
    \textup{I}$_n$ & \textup{I}$_n^*$  \\
    \textup{II} & \textup{IV}$^*$  \\
    \textup{III} & \textup{III}$^*$  \\
    \textup{IV} & \textup{II}$^*$  \\
    \textup{I}$_0^*$  & \textup{I}$_0$  \\
    \textup{II}$^*$  & \textup{IV}  \\
    \textup{III}$^*$  & \textup{III}  \\
    \textup{IV}$^*$  & \textup{II}  \\
    \hline
   \end{tabular}
\end{center}
\end{lemma}
Keeping the same notation as in the previous lemma, it is well known that if $p \neq 2$ and $p \nmid d$, then the reduction types of $E/\mathbb{Q}$ and $E^d/\mathbb{Q}$ modulo $p$ are the same. We will also need the following lemma.

\begin{lemma}\label{reductionmodulo2}
     Let $E/\mathbb{Q}$ be an elliptic curve and $d$ be a square-free number. 
     \begin{enumerate}
         \item If $E/\mathbb{Q}$ has good reduction modulo $2$, then $E^d/\mathbb{Q}$ has either good reduction or reduction of type \textup{I}$_4^*$, \textup{I}$_8^*$, \textup{II}, or \textup{II}$^*$ modulo $2$.
         \item If $E/\mathbb{Q}$ has modulo $2$ reduction of type \textup{I}$_n$ for some $n > 0$, then $E^d/\mathbb{Q}$ has modulo $2$ either reduction of type \textup{I}$_n$ or reduction of type \textup{I}$_m^*$, where $m$ is equal to either $n+4$ or $n+8$.
     \end{enumerate}
\end{lemma}
\begin{proof}
    Both of these statements are well known to the experts. We include some references here for completeness. Part $(i)$ follows from either \cite[Table I]{com} and \cite[Table II]{com}, or, alternatively, by \cite[Table 3]{kida} and keeping in mind that $E^d/\mathbb{Q}$ acquires good reduction after at most a quadratic extension. On the other hand, Part $(ii)$ follows from a theorem of Lorenzini \cite[Theorem 2.8]{LorenziniModelsofCurvesandWildRamification})
\end{proof}

We are now ready to proceed with our proofs.

 \begin{theorem}\label{theorem6}
       Let $\ell$ be equal to $19, 43, 67,$ or $163$. Let $E/\mathbb{Q}$ be an elliptic curve with a $\mathbb{Q}$-rational isogeny of prime degree $\ell$.
       \begin{enumerate}
           \item  If $p \neq 2, \ell$ is a prime, then $E/\mathbb{Q}$ has either good reduction or reduction of type \textup{I}$_0^*$ modulo $p$.
           \item  The curve $E/\mathbb{Q}$ has reduction of type \textup{III} or \textup{III}$^*$ modulo $\ell$.
           \item  The curve $E/\mathbb{Q}$ has either good reduction or reduction of type \textup{I}$_4^*$, \textup{I}$_8^*$, \textup{II}, or \textup{II}$^*$ modulo $2$.
       \end{enumerate}
\end{theorem}
 \begin{proof}
      Let $E/\mathbb{Q}$ be an elliptic curve with a $\mathbb{Q}$-rational isogeny of degree $\ell$ and let $p \neq 2,\ell$ be a prime number. We will proceed with a case by case analysis.
 
        Assume first that $\ell=19$. From \cite[Table 4]{lozanorobledo} we see that if $E/\mathbb{Q}$ is an elliptic curve with a $\mathbb{Q}$-rational isogeny of degree $19$, then its $j$-invariant is equal to $-2^{15} \cdot 3^3$. The curve $E_1$ with LMFDB label \href{https://www.lmfdb.org/EllipticCurve/Q/361/a/2}{361.a2} is a curve with the smallest conductor in the twist class with $j$-invariant $-2^{15}3^3$. Using the LMFDB database it is easy to see that $E_1/\mathbb{Q}$ has good reduction away from $19$ and that it has reduction of type III modulo $19$. Let now $E/\mathbb{Q}$ be an elliptic curve with $j(E)=-2^{15}\cdot 3^3 $. Since $j(E)=-2^{15}\cdot 3^3 \neq 0, 1728$, it follows from \cite[Corollary X.5.4.1]{aec} that there exists a square-free $d$ such that $E/\mathbb{Q}$ is isomorphic over $\mathbb{Q}$ to $E_1^d/\mathbb{Q}$. 
    
    If now $p \nmid d$, then $E_1^d/\mathbb{Q}$ and, hence, $E/\mathbb{Q}$ has good reduction modulo $p$. On the other hand, if  $p \mid d$, then it follows from Lemma \ref{results1comalada} that $E/\mathbb{Q}$ has reduction of type I$_0^*$ modulo $p$. Moreover, if $19 \nmid d$, then $E/\mathbb{Q}$ has reduction of type III modulo $19$ while if $19 \mid d$, then, by Lemma \ref{results1comalada}, we obtain that $E/\mathbb{Q}$ has reduction of type III$^*$ modulo $19$. Finally, since the curve $E_1/\mathbb{Q}$ has good reduction modulo $2$, using Part $(i)$ of Lemma \ref{reductionmodulo2}, we find that $E/\mathbb{Q}$ has either good reduction or reduction of type I$_4^*$, I$_8^*$, II, or II$^*$ modulo $2$.

    Assume that $\ell=43$. From \cite[Table 4]{lozanorobledo} we see that if $E/\mathbb{Q}$ is an elliptic curve with a $\mathbb{Q}$-rational isogeny of degree $43$, then its $j$-invariant is equal to $-2^{18}\cdot3^3\cdot5^3$. The curve $E_1$ with LMFDB label \href{https://www.lmfdb.org/EllipticCurve/Q/1849/b/2}{1849.b2} is a curve with the smallest conductor in the twist class with $j$-invariant $-2^{18}\cdot3^3\cdot5^3$. Using the LMFDB database it is easy to see that $E_1/\mathbb{Q}$ has good reduction away from $43$ and that it has reduction of type III modulo $43$. Since $j(E)=-2^{18}\cdot3^3\cdot5^3 \neq 0, 1728$, it follows from \cite[Corollary X.5.4.1]{aec} that there exists a square-free $d$ such that $E/\mathbb{Q}$ is $\mathbb{Q}$-isomorphic to $E_1^d/\mathbb{Q}$. 
    
    If now $p \nmid d$, then $E/\mathbb{Q}$ has good reduction modulo $p$. On the other hand, if  $p \mid d$, then it follows from Lemma \ref{results1comalada} that $E/\mathbb{Q}$ has reduction of type I$_0^*$ modulo $p$. Moreover, if $43 \nmid d$, then $E/\mathbb{Q}$ has reduction of type III modulo $43$ while if $43 \mid d$, then, by Lemma \ref{results1comalada}, we find that $E/\mathbb{Q}$ has reduction of type III$^*$ modulo $43$.  Finally, the curve $E_1/\mathbb{Q}$ has good reduction modulo $2$ and we find on the LMFDB database a minimal Weierstrass equation. Since the $b_2$ invariant is even, using Part $(i)$ of Lemma \ref{reductionmodulo2} we find that $E/\mathbb{Q}$ has either good reduction or reduction of type  I$_4^*$, I$_8^*$, II, or II$^*$ modulo $2$.

    Assume that $\ell=67$. From \cite[Table 4]{lozanorobledo} we see that if $E/\mathbb{Q}$ is an elliptic curve with a $\mathbb{Q}$-rational isogeny of degree $67$, then its $j$-invariant is equal to $-2^{15}\cdot3^3\cdot 5^3 \cdot11^3$. The curve $E_1$ with LMFDB label \href{https://www.lmfdb.org/EllipticCurve/Q/4489/b/2}{4489.b2} is a curve with the smallest conductor in the twist class with $j$-invariant $-2^{15}\cdot3^3 \cdot5^3 \cdot11^3$. Using the LMFDB database it is easy to see that $E_1/\mathbb{Q}$ has good reduction away from $67$ and that it has reduction of type III modulo $67$. The rest of the proof from the previous case carries over verbatim in this case. We will not reproduce the details.

    Assume that $\ell=163$. From \cite[Table 4]{lozanorobledo} we see that if $E/\mathbb{Q}$ is an elliptic curve with a $\mathbb{Q}$-rational isogeny of degree $163$, then its $j$-invariant is equal to $-2^{18}\cdot3^3\cdot 5^3\cdot 23^3 \cdot29^3$. The curve $E_1$ with LMFDB label \href{https://www.lmfdb.org/EllipticCurve/Q/26569/a/2}{26569.a2} is a curve with the smallest conductor in the twist class with $j$-invariant $-2^{18}\cdot 3^3\cdot 5^3\cdot 23^3 \cdot 29^3$. Using the LMFDB database it is easy to see that $E_1/\mathbb{Q}$ has good reduction away from $163$ and that it has reduction of type III modulo $163$. The rest of the proof from the previous case carries over verbatim in this case so we will not reproduce the details. This completes the proof of our theorem.
 \end{proof}

    \begin{theorem}\label{theorem5}
        Let $E/\mathbb{Q}$ be an elliptic curve with a $\mathbb{Q}$-rational isogeny of degree $11$.
        \begin{enumerate}
            \item If $p \neq 2, 11$ is a prime, then $E/\mathbb{Q}$ has either good reduction or reduction of type \textup{I}$_0^*$ modulo $p$.
            \item  The curve $E/\mathbb{Q}$ has reduction of type  \textup{II}, \textup{II}$^*$, \textup{III}, \textup{III}$^*$, \textup{IV}, or \textup{IV}$^*$ modulo $11$.
           \item  The curve $E/\mathbb{Q}$ has either good reduction or reduction of type \textup{I}$_4^*$, \textup{I}$_8^*$, \textup{II}, or \textup{II}$^*$ modulo $2$.
        \end{enumerate}
        
    \end{theorem}
    \begin{proof}
         Let $E/\mathbb{Q}$ be an elliptic curve with a $\mathbb{Q}$-rational isogeny of degree $11$ and let $p \neq 2,11$ be a prime number. From \cite[Table 4]{lozanorobledo} we see that if $E/\mathbb{Q}$ is an elliptic curve with a $\mathbb{Q}$-rational isogeny of degree $11$, then its $j$-invariant is equal to $-11 \cdot 131^3$, $-2^{15}$, or $-11^2$. The curves with LMFDB labels, denoted by $E_1/\mathbb{Q}$, $E_2/\mathbb{Q}$, $E_3/\mathbb{Q}$, respectively, \href{https://www.lmfdb.org/EllipticCurve/Q/121/a/2}{121.a2}, \href{https://www.lmfdb.org/EllipticCurve/Q/121/b/2}{121.b2}, \href{https://www.lmfdb.org/EllipticCurve/Q/121/c/2}{121.c2} are curves with the smallest conductors in each twist class corresponding to $j$-invariant $-11 \cdot 131^3$, $-2^{15}$, and $-11^2$, respectively. It is easy to check, using the LMFDB database, that all these curves have good reduction away from $11$. It follows from \cite[Corollary X.5.4.1]{aec} that there exists a square-free $d$ such that $E/\mathbb{Q}$ is $\mathbb{Q}$-isomorphic to either $E_1^d/\mathbb{Q}$, $E_2^d/\mathbb{Q}$, or $E_3^d/\mathbb{Q}$. If $p \nmid d$, then $E/\mathbb{Q}$ has good reduction modulo $p$. On the other hand, if  $p \mid d$, then it follows from \cite[Proposition 1]{com} that $E/\mathbb{Q}$ has reduction of type I$_0^*$ modulo $p$.

         Moreover, the curves $E_1/\mathbb{Q}$, $E_2/\mathbb{Q}$, and $E_3/\mathbb{Q}$ have reduction of type II, III, and IV modulo $11$, respectively. Therefore, we see from Lemma \ref{results1comalada} that $E/\mathbb{Q}$ has reduction of type II, II$^*$, III, III$^*$, IV, or IV$^*$ modulo $11$. Finally, since $E_1/\mathbb{Q}$, $E_2/\mathbb{Q}$, and $E_3/\mathbb{Q}$ all have good reduction modulo $2$, using Lemma \ref{reductionmodulo2} we find that $E/\mathbb{Q}$ has either good reduction or reduction of type  I$_4^*$, I$_8^*$, II, or II$^*$ modulo $2$.
    \end{proof}
   
     \begin{example}\label{example11overq}
        Consider the elliptic curves with LMFDB labels \href{https://www.lmfdb.org/EllipticCurve/Q/121/a/2}{121.a2}, \href{https://www.lmfdb.org/EllipticCurve/Q/121/a/1}{121.a1}, \href{https://www.lmfdb.org/EllipticCurve/Q/121/b/2}{121.b2}, \href{https://www.lmfdb.org/EllipticCurve/Q/121/b/1}{121.b1}, \href{https://www.lmfdb.org/EllipticCurve/Q/1089/c/2}{1089.c2}, and \href{https://www.lmfdb.org/EllipticCurve/Q/1089/c/1}{1089.c1}. Those curves have $\mathbb{Q}$-rational isogeny of degree $11$ and reduction modulo $11$ of type II, II$^*$, III, III$^*$, IV, and IV$^*$, respectively.
    \end{example}

    \begin{theorem}\label{theorem8}
        Let $E/\mathbb{Q}$ be an elliptic curve with a $\mathbb{Q}$-rational isogeny of degree $17$.
        \begin{enumerate}
           \item If $p \neq 2, 5, 17$ is a prime, then $E/\mathbb{Q}$ has either good reduction or reduction of type \textup{I}$_0^*$ modulo $p$.
           \item The curve $E/\mathbb{Q}$ has reduction of type \textup{I}$_1$, \textup{I}$_{17}$, \textup{I}$_{5}^*$, \textup{I}$_{9}^*$, \textup{I}$_{21}^*$, or \textup{I}$_{25}^*$ modulo $2$.
           \item The curve $E/\mathbb{Q}$ has reduction of type \textup{III} or \textup{III}$^*$ modulo $5$.
           \item The curve $E/\mathbb{Q}$ has reduction of type \textup{II}, \textup{II}$^*$, \textup{IV},  or \textup{IV}$^*$ modulo $17$.
       \end{enumerate}
    \end{theorem}
    \begin{proof}
         Let $E/\mathbb{Q}$ be an elliptic curve with a $\mathbb{Q}$-rational isogeny of degree $17$. From \cite[Table 4]{lozanorobledo} we see that if $E/\mathbb{Q}$ is an elliptic curve with a $\mathbb{Q}$-rational isogeny of degree $17$, then its $j$-invariant is equal to $-\frac{17^2 \cdot 101^3}{2}$ or $-\frac{17 \cdot 373^3}{2^{17}}$. The curves with LMFDB labels \href{https://www.lmfdb.org/EllipticCurve/Q/14450/b/2}{14450.b2} and \href{https://www.lmfdb.org/EllipticCurve/Q/14450/b/1}{14450.b1}, denoted by $E_1/\mathbb{Q}$ and $E_2/\mathbb{Q}$, respectively, are curves with the smallest conductors in each twist class corresponding to $j$-invariant $-\frac{17^2 \cdot 101^3}{2}$ and $-\frac{17 \cdot 373^3}{2^{17}}$, respectively. Using the LMFDB database it is easy to see that each of those curves has good reduction away from $2, 5$ and $17$. Therefore, proceeding similarly as in the proofs of the previous theorems in this section, we can show that $E/\mathbb{Q}$ has either good reduction or reduction of type I$_0^*$ modulo $p$, for $p$ a prime such that $p \neq 2, 5, 17$. The curves $E_1/\mathbb{Q}$ and $E_2/\mathbb{Q}$ have reduction of type III modulo $5$. Therefore, $E/\mathbb{Q}$ can only have reduction of type III or III$^*$ modulo $5$.

         Moreover, using the LMFDB database we see that the curves $E_1/\mathbb{Q}$ and $E_2/\mathbb{Q}$ have reduction of type IV and IV$^*$ modulo $17$, respectively. Thus, using Lemma \ref{results1comalada} we find that the curve $E/\mathbb{Q}$ has reduction of type II, II$^*$, IV,  or IV$^*$ modulo $17$. Finally, the curves $E_1/\mathbb{Q}$ and $E_2/\mathbb{Q}$ have reduction of type I$_1$ and I$_{17}$ modulo $2$, respectively. Therefore, using Lemma \ref{reductionmodulo2} we find that $E/\mathbb{Q}$ has reduction of type I$_1$, I$_{17}$, I$_{5}^*$, I$_{9}^*$, I$_{21}^*$, or I$_{25}^*$ modulo $2$.
         
    \end{proof}

  \begin{theorem}\label{theorem7}
       Let $E/\mathbb{Q}$ be an elliptic curve with a $\mathbb{Q}$-rational isogeny of degree $37$.
       \begin{enumerate}
           \item If $p \neq 2, 5, 7$ is a prime, then $E/\mathbb{Q}$ has either good reduction or reduction of type \textup{I}$_0^*$ modulo $p$.
           \item The curve $E/\mathbb{Q}$ has either good reduction or reduction of type \textup{I}$_4^*$, \textup{I}$_8^*$, \textup{II}, or \textup{II}$^*$ modulo $2$.
           \item The curve $E/\mathbb{Q}$ has reduction of type \textup{III} or \textup{III}$^*$ modulo $5$.
           \item The curve $E/\mathbb{Q}$ has reduction of type \textup{II} or \textup{IV}$^*$ modulo $7$.
       \end{enumerate}
    \end{theorem}
    \begin{proof}
        Let $E/\mathbb{Q}$ be an elliptic curve with a $\mathbb{Q}$-rational isogeny of degree $37$.  From \cite[Table 4]{lozanorobledo} we see that if $E/\mathbb{Q}$ is an elliptic curve with a $\mathbb{Q}$-rational isogeny of degree $37$, then its $j$-invariant is equal to $-7 \cdot 11^3$ or $-7 \cdot 137^3 \cdot 2083^3$. The curves with LMFDB labels \href{https://www.lmfdb.org/EllipticCurve/Q/1225/b/2}{1225.b2} and \href{https://www.lmfdb.org/EllipticCurve/Q/1225/b/1}{1225.b1}, denoted by $E_1/\mathbb{Q}$ and $E_2/\mathbb{Q}$, respectively, are curves with the smallest conductors in each twist class corresponding to $j$-invariant $-7 \cdot 11^3$ and $-7 \cdot 137^3 \cdot 2083^3$, respectively. Using the LMFDB database it is easy to see that each of those curves has good reduction away from $5$ and $7$. Therefore, proceeding similarly as in the proofs of the previous theorems in this section, we can show that $E/\mathbb{Q}$ has either good reduction or reduction of type I$_0^*$ modulo $p$, for $p$ a prime such that $p \neq 5,7$. 

        On the other hand, the curves $E_1/\mathbb{Q}$ and $E_2/\mathbb{Q}$ have reduction of type III modulo $5$. Therefore, $E/\mathbb{Q}$ can only have reduction of type III or III$^*$ modulo $5$. Finally, both curves $E_1/\mathbb{Q}$ and $E_2/\mathbb{Q}$ have reduction of type II modulo $7$. Therefore, $E/\mathbb{Q}$ can only have reduction of type II or IV$^*$ modulo $7$. Finally, since $E_1/\mathbb{Q}$ has good reduction modulo $2$, using Lemma \ref{reductionmodulo2} we find that $E/\mathbb{Q}$ has either good reduction or reduction of type  I$_4^*$, I$_8^*$, II, or II$^*$ modulo $2$.
    \end{proof}
\begin{remark}
   Given any primes $p, \ell$ with $p \neq 2$ and any reduction type $T$ that appears in Theorem \ref{theorem6}, \ref{theorem5}. \ref{theorem8}, or \ref{theorem7}, then by using an appropriate quadratic twist one can find an elliptic curve with a $\mathbb{Q}$-rational isogeny of degree $\ell$ and reduction type $T$ modulo $p$. For example, suppose we are looking for an elliptic curve $E/\mathbb{Q}$ with a $\mathbb{Q}$-rational isogeny of degree $17$ and reduction of Kodaira type III$^*$ modulo $5$. The curve $E_1/\mathbb{Q}$ with LMFDB label \href{https://www.lmfdb.org/EllipticCurve/Q/14450/b/2}{14450.b2}, which appears in the proof of Theorem \ref{theorem8}, is an elliptic curve with a $\mathbb{Q}$-rational isogeny of degree $17$ and reduction of Kodaira type III modulo $5$. Therefore, it follows from Lemma \ref{results1comalada} that the quadratic twist $E_1^{5}/\mathbb{Q}$ is an elliptic curve with reduction of Kodaira type III$^*$ modulo $5$ and it has a $\mathbb{Q}$-rational isogeny of degree $17$. Thus we have found an example with the required properties. We can proceed in a similar way for the other choices of primes $p$, $\ell$, and Kodaira types $T$.
\end{remark}

\bibliographystyle{plain}
\bibliography{bibliography.bib}

\end{document}